\documentclass{amsart}[10pt]
\usepackage[a4paper,margin=1.2in]{geometry}  
\usepackage{amsmath,amsfonts,amssymb}               
\usepackage{amsthm}
\usepackage[utf8]{inputenc}
\usepackage{enumerate}
\usepackage{marginnote}
\usepackage{tikz}
\usepackage{tikz-cd}
\usetikzlibrary{decorations.markings}
\tikzset{degil/.style={
            decoration={markings,
            mark= at position 0.5 with {
                  \node[transform shape] (tempnode) {$\backslash$};
                  \draw[thick] (tempnode.north east) -- (tempnode.south west);
                  }
              },
              postaction={decorate}
}
}
\usetikzlibrary{decorations.text, calc, matrix, arrows, decorations.pathreplacing}
\usepackage[all]{xy}
\usepackage{forloop}
\usepackage{soul}
\usepackage{color}
\usepackage{leftidx}
\usepackage{hyperref}
\usepackage[capitalise]{cleveref}

\makeatletter
\newtheorem*{rep@theorem}{\rep@title}
\newcommand{\newreptheorem}[2]{%
\newenvironment{rep#1}[1]{%
 \def\rep@title{#2 \ref{##1}}%
 \begin{rep@theorem}}%
 {\end{rep@theorem}}}
\makeatother

\newtheorem{thm}{Theorem}[section]
\newtheorem{lemma}[thm]{Lemma}
\newtheorem{prop}[thm]{Proposition}
\newtheorem{cor}[thm]{Corollary}

\newreptheorem{theorem}{Theorem}
\newreptheorem{lemma}{Lemma}

\theoremstyle{definition} 
\newtheorem{remark}[thm]{Remark}

\newcommand\mydef{\mathrel{\overset{\makebox[0pt]{\mbox{\normalfont\tiny def}}}{=}}}
\newcommand{\comp}{\kern0.5ex\vcenter{\hbox{$\scriptstyle\circ$}}\kern0.5ex}

\newcommand{\bb}[1]{\mathbb{#1}}

\newcommand{\cO}{\mathcal O}

\newcommand{\cT}{\mathcal T}

\newcommand{\cH}{\mathcal H}

\newcommand{\cX}{\mathcal X}
\newcommand{\cY}{\mathcal Y}

\newcommand{\isom}{\simeq}

\newcommand{\ideal}[1]{{\mathfrak{#1}}}

\DeclareMathOperator{\id}{id}
\DeclareMathOperator{\PP}{\bb{P}}
\DeclareMathOperator{\ZZ}{\bb{Z}}

\DeclareMathOperator{\FF}{\bb{F}}
\DeclareMathOperator{\LL}{\bb{L}}

\DeclareMathOperator{\len}{len}
\DeclareMathOperator{\Ob}{Ob}
\DeclareMathOperator{\Tan}{T}

\DeclareMathOperator{\Bl}{Bl}
\DeclareMathOperator{\proddef}{prod}
\DeclareMathOperator{\Ann}{Ann}
\DeclareMathOperator{\SL}{SL}

\newcommand{\Ext}{\mathrm{Ext}}

\newcommand{\Spec}{\mathrm{Spec}}

\newcommand{\Cot}{\mathbb{L}}

\newcommand{\Def}{\mathsf{Def}}
\newcommand{\Art}{\mathbf{Art}}

\newcommand{\ra}{\longrightarrow}

\newcommand{\wt}[1]{\widetilde{#1}}  
\newcommand{\ol}[1]{\overline{#1}} 

\begin{document}
\title{Some elementary examples of non-liftable varieties}
\author{Piotr Achinger}
\address{Banach Center, Instytut Matematyczny PAN, Śniadeckich 8, Warsaw, Poland}
\email{pachinger@impan.pl}

\author{Maciej Zdanowicz}
\address{Wydział Matematyki, Informatyki i Mechaniki UW, Banacha 2, Warsaw, Poland}
\email{mez@mimuw.edu.pl}

\begin{abstract}
We present some simple examples of smooth projective varieties in positive characteristic, arising from linear algebra, which do not admit a lifting neither to characteristic zero, nor to the ring of second Witt vectors.  Our first construction is the blow-up of the graph of the Frobenius morphism of a homogeneous space. The second example is a blow-up of $\PP^3$ in a `purely characteristic-$p$' configuration of points and lines. 
\end{abstract}

\maketitle

\section{Introduction}\label{sec:intro}

Various theorems in modern algebraic geometry are proved using characteristic $p$ methods along the following lines. Given a complex algebraic variety $X$, one reduces the variety mod $p$, exploits the Frobenius morphism on the reduction $X_p$, and deduces statements about the original $X$. Similarly, characteristic zero (particularly complex analytic) methods are employed to study varieties in positive characteristic. The main technical obstacle is that, while every characteristic zero variety can be reduced mod $p$, not every variety in positive characteristic arises as the reduction mod $p$ of a variety in characteristic zero. The first example of such a variety was given by Serre \cite{serre}. 

It turns our that for many purposes, one does not need to lift a given variety all the way to characteristic zero, and it suffices to have a lifting modulo $p^2$. For example, Deligne and Illusie \cite{deligne_illusie} showed that for a smooth variety $X$ over a perfect field $k$ of characteristic $p>\dim X$ admitting a lifting to $W_2(k)$ (the ring of Witt vectors of length 2), the Hodge--de Rham spectral sequence degenerates, and the Kodaira vanishing theorem holds. More recently, Langer \cite{langer} showed that the logarithmic Bogomolov--Miyaoka--Yau inequality holds for surfaces liftable to $W_2(k)$ (as long as $p>2$). Counterexamples to Kodaira vanishing in positive characteristic given by Raynaud \cite{raynaud} give the first example of varieties which do not lift to $W_2(k)$. Subsequently, a rational example was given by Lauritzen and Rao \cite{lauritzen_rao}. In the positive direction, it is known that every Frobenius split variety lifts to $W_2(k)$ \cite[Proposition 8.4]{langer_higgs}.

In this paper, we construct new examples of smooth projective varieties that do not admit lifts neither to characteristic zero, nor to $W_2(k)$ (some of them do not even lift to any ring $A$ with $pA\neq 0$).  However it turns out that they avoid standard characteristic $p$ pathologies, in particular they satisfy the following
\begin{description}
\item[Good properties]
\end{description}
\begin{enumerate}
	\item they are smooth, projective, rational, and simply connected,
	\item their classes in the Grothendieck ring of varieties are polynomials in the Lefschetz motive $\mathbb{L}=[\mathbb{A}^1]$ with non-negative integer coefficients, 
  	\item their $\ell$-adic integral cohomology rings are generated by algebraic cycles,
	\item their integral crystalline cohomology groups are torsion-free $F$-crystals,
	\item their Hodge--de Rham and conjugate spectral sequences degenerate, they are ordinary in the sense of Bloch--Kato, and of Hodge--Witt type (cf. \S\ref{sec:ordinarity} for the relevant definitions).
\end{enumerate}
Since our constructions are very simple, we try to aim our exposition at non-experts, and go for elementary arguments whenever possible. 

The first construction is given by the blow-up of the two-fold self product of a suitable projective homogeneous space $\neq\mathbb{P}^n$ along the graph of its Frobenius morphism.  The easiest examples of such homogeneous spaces being the three-dimensional complete flag variety $\SL_3/B$ (isomorphic to the incidence variety $\{x_0y_0 + x_1y_1 + x_2y_2 = 0\} \subset \PP^2 \times \PP^2$) and the three-dimensional smooth quadric hypersurface $Q=\{x_0^2 + x_1 x_2+x_3x_4 = 0\}\subseteq \mathbb{P}^4$ 
, the smallest non-liftable examples given by the construction are six-dimensional with Picard numbers five and three, respectively.  The proof of the above good properties and non-liftability is given in \cref{thm:graph}.


The second construction is the following.  Let $X$ be the variety obtained from $\mathbb{P}^3$ by (1) blowing up all $\mathbb{F}_p$-rational points, and (2) blowing up the strict transforms of all lines connecting $\mathbb{F}_p$-rational points. Then $X$ satisfies (1)--(5) above, but does not admit a lift to any ring $A$ with $pA\neq 0$.  The proofs are presented in \cref{thm:config}.

Both in \cref{thm:graph} and \cref{thm:config}, the proofs of non-liftability use the key observation (cf. \cite{liedtke_satriano}, and \cref{lem:blow-up} below) that if the blow-up of a smooth variety $X$ along a smooth subvariety $Z$ lifts, then both $X$ and $Z$ lift. In \cref{thm:graph}, if $X$ was liftable, the homogeneous space $Y$ would be liftable together with Frobenius, which is known to be impossible by the work of Paranjape--Srivinas \cite{paranjape_srinivas} (for lifts to characteristic zero) and Buch--Thomsen--Lauritzen--Mehta \cite{frobenius_flag_varieties} (for lifts to $W_2(k)$). In \cref{thm:config}, we show that the liftability of $X$ would imply the liftability of the arrangement of all $\mathbb{F}_p$-rational points in $\mathbb{P}^2$ preserving the incidence relations; thus non-liftability is established by means of elementary linear algebra. The properties (1)--(5) in both theorems are established quite easily using standard formulas expressing the cohomology of a blow-up which we recall in \S\ref{sec:cohomology}.

\subsection{Notation} \label{sec:notation}

Throughout $k$ denotes a perfect field of characteristic $p>0$.  For any $k$-scheme $X$ by $X^{(1)}$ we denote the Frobenius pullback $X^{(1)} \mydef X \times_{\Spec(k),F} \Spec(k)$ and by $F_{X/k} : X \to X^{(1)}$ the relative Frobenius of $X$ over $k$.  We say that a scheme $X/k$ admits a $W_2(k)$-lifting if there exists a flat $W_2(k)$-scheme $\wt{X}$ such that $\wt{X} \times_{\Spec(W_2(k))} \Spec(k) \isom X$.  Finally, we say that a scheme $X$ lifts to $W_2(k)$ compatibly with Frobenius if there exists a $W_2(k)$-lifting $\wt{X}$ of $X$ together with a morphism $\wt{F_{X/k}} : \wt{X} \to \wt{X}^{(1)} \mydef \wt{X} \times_{\Spec(W_2(k)),\sigma} \Spec(W_2(k))$ restricting to the relative Frobenius morphism $F_{X/k} : X \to X^{(1)}$.  For schemes defined over the field $\FF_p$ the absolute Frobenius morphism is in fact $\FF_p$-linear and therefore the relative Frobenius morphism can be interpreted as an endomorphism $F_X : X \to X^{(1)} \isom X$.   

By $\Cot_{X/k}$ we denote the cotangent complex of a scheme $X$ over $k$.  Moreover, by $\Def_X$ we mean the deformation functor of $X$, that is, a covariant functor from the category $\Art_{W(k)}(k)$ of Artinian local $W(k)$-algebras with residue field $k$ to the category of sets defined by the formula:
\[
\Art_{W(k)}(k) \ni (A,\ideal{m}_A) \mapsto \Def_X(A) \mydef \left\{ 
\begin{gathered}
\text{ isomorphism classes of flat } \\
\text{ deformations of $X$ over $\Spec(A)$ }
\end{gathered}
\right\}.
\]
Similarly, if $Z = \{Z_i \}_{i\in I}$ is a family of closed subschemes of $X$ indexed by a preorder $I$ (i.e., a set with a reflexive and transitive binary relation), such that $Z_i$ is a closed subscheme of $Z_j$ whenever $i\leq j$ (in other words, $I$ is a small category whose morphism sets have at most one element, and $Z$ is a functor from $I$ to the category of closed subschemes of $X$), we denote by $\Def_{X, Z}$ the functor of flat deformations of $X$ together with compatible embedded deformations of the $Z_i$, preserving the inclusion relations given by the relation $\leq$.
If $f:X\to Y$ is a map of $k$-schemes, we denote by $\Def_f$ the functor of flat deformations of $X$, and $Y$ along with a deformation of $f$.


\medskip

\noindent {\bf Acknowledgements.} We thank Andre Chatzistamatiou, Adrian Langer, Christian Liedtke, Vasudevan Srinivas, Bernd Sturmfels, and Jarosław Wiśniewski for helpful discussions. The first author was supported by NCN OPUS grant number UMO-2015/17/B/ST1/02634. The second author was supported by NCN PRELUDIUM grant number UMO-2014/13/N/ST1/02673. This work was partially supported by the grant 346300 for IMPAN from the Simons Foundation and the matching 2015--2019 Polish MNiSW fund. 

\section{The first construction}

We fix a semisimple algebraic group $G$ over $k=\FF_p$, a reduced parabolic subgroup $P\subseteq G$, and set $Y=G/P$. We assume that either $G$ is of type $A$ and $Y$ is not a projective space, or that $P$ is contained in a maximal parabolic subgroup as listed in \cite[4.3.1--4.3.7]{frobenius_flag_varieties} (these are the cases in which we know that $Y$ does not lift to $W_2(k)=\ZZ/p^2\ZZ$ together with Frobenius). For example, $Y$ could be the Grassmannian ${\rm Gr}(n, k)$ ($1<k<n-1$) or the full flag variety $SL_n/B$ ($n\geq 3$, $B=$ upper-trianular matrices), or a smooth quadric hypersurface in $\PP^n$, $n\geq 4$. Presumably all homogeneous spaces which are not toric (i.e., not a product of projective spaces) do not admit a lift to $W_2(k)$ together with Frobenius. 

\begin{thm}\label{thm:graph}
Let $\Gamma_F \subseteq Y\times Y$ be the graph of the Frobenius morphism $F_Y : Y \to Y$.  Let $X = \Bl_{\Gamma_{F}}(Y \times Y)$ be the blow-up of $Y\times Y$ along $\Gamma_{F}$, and $X' = \Bl_{\Delta_Y}(Y \times Y)$ the blow-up of $Y\times Y$ along the diagonal. Then $X$ and $X'$ share the good properties of \cref{sec:intro} and moreover:
\begin{enumerate}[a)]
	\item they are `\'etale homeomorphic,' i.e., their \'etale sites are equivalent; 
  \item their $\ell$-adic integral cohomology rings are isomorphic as Galois representations
	\item their integral crystalline cohomology groups are isomorphic torsion-free $F$-crystals.
\end{enumerate}
However, $X'$ admits a projective lift to $W(k)$, while $X$ lifts neither to characteristic zero (even formally), nor to $W_2(k)$.
\end{thm}
\begin{proof}
Good property (1) follows from Bruhat decomposition and the birational invariance of the \'{e}tale fundamental group of smooth varieties.  Properties (2)--(5) follow from the results of sections \ref{sec:cohomology}--\ref{sec:ordinarity}. Property (a) follows from the existence of the following cartesian diagram:
\begin{displaymath}
\xymatrix{
	X \ar[r]^u \ar[d]_{f} & X' \ar[d]^{f'} \\
	Y \times Y \ar[r]_{\id \times F_Y} & Y \times Y,}
\end{displaymath}
where $f$ and $f'$ are the respective blow-up maps. Indeed, the Frobenius map $F_{X'}:X'\to X'$ and the composition $(F_Y\times \id)\circ f' :X'\to Y\times Y$ yield a map $v:X'\to X$ making the diagram  
\[ 
  \xymatrix{
    X' \ar@/^1.2em/[rr]^{F_{X'}} \ar[r]_v \ar[d]_{f'} & X \ar@/^1.2em/[rr]^{F_{X}} \ar[r]_u \ar[d]_{f} & X' \ar[r]_v \ar[d]_{f'} &   X \ar[d]_{f}  \\
    Y\times Y \ar[r]^{F_Y\times \id} \ar@/_1.2em/[rr]_{F_{Y\times Y}}  & Y\times Y \ar[r]^{\id \times F_Y} \ar@/_1.2em/[rr]_{F_{Y\times Y}} & Y\times Y \ar[r]^{F_Y\times \id}  & Y\times Y
  }
\]
commute. In particular, $v\circ u=F_X$ and $u\circ v = F_{X'}$. Since $F_X$ and $F_{X'}$ are \'etale homeomorphisms by \cite[XIV=XV \S{}1 $n^\circ$ 2, Pr. 2(c)]{SGA5}, $u$ and $v$ are \'etale homeomorphisms as well. Property (b) follows from (a). Finally, property (c) follows from the blow-up formula \S\ref{sec:cohomology}. We remark that the crystalline cohomology algebras $H^*_{\rm cris}(X/W)$ and $H^*_{\rm cris}(X'/W)$ are not isomorphic, but become so after inverting $p$.

We now prove that $X'$ lifts to $W(k)$ projectively and that $X$ does not lift either to $W_2(k)$ or any ramified extension of $W(k)$.  For the first claim, we observe that $Y$ lifts to a projective scheme $\cY$ over $W(k)$ and consequently $\cX' = \Bl_{\Delta_\cY}(\cY \times_{W(k)} \cY)$ is a projective lifting of $X'$.  We now proceed to the second claim.  We begin with a proposition addressing Frobenius liftability of homogeneous spaces and describing their cohomological properties necessary to apply deformation theoretic results stated in \S\ref{sec:deformation}.
  
\begin{prop}\label{example:properties_grassmanian}
Let $Y$ be a homogeneous space over $k$ of a semisimple algebraic group $G$ not isomorphic to any projective space.  Then, $Y$ does not admit a $W_2(k)$--lifting compatible with Frobenius.  Moreover, it satisfies $H^1(Y,\cT_Y) = 0$ and $H^i(Y,\cO_Y) = 0$ for $i > 0$.  
\end{prop}
\begin{proof}
For the part of the proof concerning Frobenius liftability see \cite[Theorem 6]{frobenius_flag_varieties}.  Vanishing of $H^1(Y,\cT_Y)$ follows from \cite[Th\'{e}oreme 2]{demazure}.  Finally, $H^i(Y,\cO_Y) = 0$ is the consequence of Kempf vanishing (i.e., a characteristic $p$ analogue of the Borel--Weil--Bott theorem) as $0$ is a dominant weight for the parabolic subgroup of $G$ corresponding to $Y$.
\end{proof}

We show $X$ does not lift to $W_2(k)$.  Assume the contrary, i.e., that there exists a $W_2(k)$-lifting of $\Bl_{\Gamma_{F_{Y/k}}}(Y \times Y)$.  By \cref{lem:lifting_graph} there exist two liftings $\wt{Y}$ and $\wt{Y}'$ of $Y$ together with a lifting $\wt{F_{G/k}} : \wt{Y} \to \wt{Y}'$ of $F_{Y/k} : Y \to Y$.  However, by the property $H^1(Y,\cT_Y) = 0$ the homogeneous space $Y$ is rigid, which implies that the lifting $\wt{Y}'$ is isomorphic to $\wt{Y}$.  This implies that $Y$ is $W_2(k)$-liftable compatibly with Frobenius, which contradicts \cref{example:properties_grassmanian}. 

Finally, we address characteristic $0$ non-liftability of $X$.  Again, we reason by contradiction.  Any characteristic $0$ lifting of $X$ induces a formal lifting of $X$ which be \cref{lem:lifting_graph} and rigidity of $Y$ gives a formal lifting of a non-trivial endomorphism $F_Y : Y \to Y$.  By the Grothendieck algebraization theorem the formal lifting of the finite morphism $F_Y$ extends to an algebraic lifting which contradicts the final result of \cite{paranjape_srinivas} stating that homogeneous spaces in characteristic $0$ not isomorphic to products of projective spaces admit no non-trivial endomorphisms.
\end{proof}

\section{Second construction}

We work over an algebraically closed field $k$ of characteristic $p$. Let $P=\PP^3(\FF_p) \subseteq \PP^3_k$ be the set of all $\#\PP^3(\FF_p) = 1+p+p^2+p^3$ $\FF_p$-rational points, let $Y=\Bl_P \PP^3_k$, and let $L$ be the set of $\#G(2,4)(\FF_p) = \binom{1+p+p^2+p^3}{2}/\binom{1+p}{2} = 1+p+2p^2+p^3+p^4$ lines in $\PP^3_k$ meeting $P$ at least twice.  Finally, let $\tilde{L}\subseteq Y$ be the set of the strict transforms of all elements of $L$, and let $X=\Bl_{\tilde L} Y$. 

\begin{thm}\label{thm:config}
The threefold $X$ has the good properties from the introduction, but does not admit a lift to any ring $A$ with $pA\neq 0$.
\end{thm}

For the good properties (1)--(5), we argue exactly as in the previous section. The proof that $X$ does not deform to any algebra $A$ with $pA\neq 0$ consists of the following three propositions.

\begin{prop}
Let $A$ be an object of $\Art_{W(k)}(k)$, and suppose that $X$ lifts to $A$. Then $\PP^3_k$ lifts to $A$ together with all $\FF_p$-rational points and lines, preserving the incidence relations. 
\end{prop}

\begin{proof}
Let $E$ be the set consisting of the preimages in $Y$ of the elements of $P$, $F$ the set of preimages in $X$ of the elements of $\tilde L$. Finally, let $Q = (\bigcup \tilde L) \cap (\bigcup E)$ (treated as a set of points).  We have the following chain of natural transformations between various deformation functors:
\[
\xymatrix{
	\Def_X & \Def_{X,F} \ar[l]_\isom\ar[r] & \Def_{Y,\tilde{L}}  & \Def_{Y,\tilde{L}\cup E} \ar[l]_\isom  & \Def_{Y,\tilde{L} \cup E \cup Q}  \ar[l]_\isom \ar[r] & \Def_{\PP^3_k,L \cup P} .}
\]
We remind the reader of our convention (cf. \S\ref{sec:notation}) that for a family of closed subschemes $Z=\{Z_i\}_{i\in I}$ of a scheme $X$ indexed by a preorder $I$, $\Def_{X, Z}$ is the functor of deformations of $X$, together with embedded deformations of $Z_i$, preserving the inclusion relations $Z_i\subseteq Z_{i'}$ for $i\leq i'$. Above, we give the families $F, \tilde L, \tilde L\cup E$ the trivial order, and order $\tilde L\cup E\cup Q$ and $L\cup P$ by inclusion. In particular, the functor $ \Def_{Y,\tilde{L} \cup E \cup Q}$ parametrizes deformations of $Y$ together with the strict transforms of the $\FF_p$-rational lines (i.e., $\tilde L$) and the preimages of the $\FF_p$-rational points (i.e., $E$) in $\PP^3_k$ such that their mutual intersections are flat over the base (i.e., induce a compatible embedded deformation of $Q$). Similarly, $\Def_{\PP^3_k, P\cup L}$ is the functor of deformations of $\PP^3_k$ together with all the $\FF_p$-rational points and lines, preserving the incidence relations.  We discuss the maps in this chain below.

The maps $\Def_{X, F}\to \Def_X$, $ \Def_{Y,\tilde{L}\cup E}\to  \Def_{Y,\tilde{L}}$, and $\Def_{Y,\tilde{L}\cup E\cup Q} \to \Def_{Y,\tilde{L} \cup E}$ are the forgetful transformations. 
The first two are isomorphisms by \cref{lem:blow-up}(2), and the last map is an isomorphism by \cref{cor:reg_seq} of \cref{lem:nonzero} applied to the local equations of $E$ and $\wt{L}$.

The maps $\Def_{X,F} \to \Def_{Y,\tilde{L}}$ and $\Def_{Y,\tilde{L} \cup E \cup Q}\to \Def_{\PP^3_k,L \cup P}$ are the maps of \cref{lem:blow-up}(1). For the latter, strictly speaking, \cref{lem:blow-up}(1) yields a map $\Def_{Y,\tilde{L} \cup E \cup Q}\to \Def_{\PP^3_k, Z}$, where $Z=\{Z_s\}_{s\in S}$ is the `image' of $\tilde L\cup E\cup Q$, defined as follows. Let $S=L\sqcup P\sqcup K$ where $K=\{(x, \ell)\in P\times L\,:\, x\in\ell\}$, given the ordering whose nontrivial relations are $(\ell, x)\leq \ell$ and $(\ell , x) \leq x$ for $x\in P, \ell\in L, (x, \ell)\in K$. Then set $Z_\ell = \ell$ for $\ell\in L$, $Z_x = x$ for $x\in P$, and $Z_{(x, \ell)} = x$ for $(x, \ell)\in K$. For an algebra $A$, an element of $\Def_{\PP^3_k, Z}$ is thus given by a deformation of $\PP^3_k$ together with deformations of the $\ell\in L$, $x\in X$, and $Z_{(x, \ell)}=x$ for $(x, \ell)\in K$, preserving the relations $\tilde Z_{(x, \ell)} \subseteq \tilde x$ and $\tilde Z_{(x, \ell)}\subseteq \tilde \ell$ for $(x, \ell)\in K$ (here the tildes mean the corresponding deformations over $A$). But each $x$ is a point, so $\tilde Z_{(x, \ell)}\subseteq \tilde x$ implies $\tilde Z_{(x, \ell)}= \tilde x$, and the deformation of $(\PP^3_k, Z)$ simplifies to a deformation of $(\PP^3_k, L\cup P)$ preserving the incidence relations. Thus $\Def_{\PP^3_k, Z}$ can be identified with $\Def_{\PP^3_k, L\cup P}$.
\end{proof}

\begin{remark}
Since we will have to deal with a little bit of elementary projective geometry and matroid representability over arbitrary rings, let us fix some conventions. Let $A$ be a local ring with residue field $k$. A projective $n$-space $\PP$ over $A$ is an $A$-scheme isomorphic to $\PP^n_A$, and a $d$-dimensional linear subspace $L$ of $\PP$ is a closed subscheme of $\PP$ which is flat over $A$ and such that $L\otimes k$ is a linear subspace of $\PP\otimes k$. Zero-dimensional linear subspaces of $\PP$ can be identified with the set $\PP(A)$. If $x, y\in \PP(A)$ are points whose images in $\PP(k)$ are distinct, there exists a unique line (i.e., a one-dimensional linear subspace) $\ell(x, y)$ containing both $x$ and $y$. We say that points $x,y,z$ are collinear (resp. coplanar) if they lie on one line (resp. $2$-dimensional subspace).   If $x_0, \ldots, x_n, z$ are points whose images in $\PP(k)$ are in general position, there exists a unique isomorphism $\phi:\PP\to \PP^n_A$ such that $\phi(x_i) = e_i:=(0:\ldots:0:1:0:\ldots:0)$ (with $1$ on the $i$-th coordinate) and $\phi(z) = f:=(1:\ldots:1)$. In particular, if $A\in \Art_{W(k)}(k)$, and $S$ is a configuration of linear subspaces of $\PP^n_k$ ordered by inclusion, containing the points $e'_i=(0:\ldots:0:1:0:\ldots:0)$ and $f'=(1:\ldots 1)$, we can identify the deformation functor $\Def_{\PP^n_k, S}(A)$ with the set of all families of linear subspaces $\tilde S$ in $\PP^n_A$ which yield the given $S$ upon restriction to $k$, and such that $\tilde e'_i = e_i$ and $\tilde f'= f$.
\end{remark}

\begin{prop}
Suppose that $\PP^3_k$ lifts to an Artinian $W(k)$-algebra $A$ together with all $\FF_p$-rational points, preserving collinearity. Then the same holds for $\PP^2_k$.
\end{prop}

\begin{proof}
The key observation is that coplanarity is also preserved, i.e., that $\Def_{\PP^3_k, L\cup P}=\Def_{\PP^3_k, H\cup L\cup P}$, where $H$ denotes the set of all $\FF_p$-rational hyperplanes in $\PP^3_k$ (with $H\cup L\cup P$ ordered by inclusion). Indeed, let $A$ be an object of $\Art_{W(k)}(k)$, and suppose we are given an element of $\Def_{\PP^3_k, L\cup P}(A)$, which by simple rigidification (e.g., the requirement that the points $(1:0:0)$, $(0:1:0)$, $(0:0:1)$, and $(1:1:1)$ do not deform) can be identified with a configuration of points $\tilde x$ and lines $\tilde \ell$ in $\PP^3_A$, indexed by $P$ and $L$ respectively, such that $\tilde x\subseteq \tilde \ell$ whenever $x\in \ell$. To get an element of $\Def_{\PP^3_k, H\cup L\cup P}$, it suffices to show that whenever $x_1, x_2, x_3, x_4\in P$ is a quadruple of coplanar points, the points $\tilde x_1, \tilde x_2, \tilde x_3, \tilde x_4 \in \PP^3(A)$ are coplanar. If two of the points $x_i$ coincide, there is nothing to show, and similarly if all four lie on a line. Otherwise, let $\ell_{12}= \ell(x_1, x_2)$ and $\ell_{34} = \ell(x_3, x_4)$, then $\tilde \ell_{12} = \ell(\tilde x_1, \tilde x_2)$ and $\tilde \ell_{34} = \ell(\tilde x_3, \tilde x_4)$. Since the $x_i$ are coplanar, the lines $\ell_{12}$ and $\ell_{34}$ intersect in a unique point $y\in P$. Then $\tilde y\in \tilde \ell_{12}\cap \tilde \ell_{34} = \ell(\tilde x_1, \tilde x_2) \cap  \ell(\tilde x_3, \tilde x_4)$. Thus the hyperplane through $\tilde y, \tilde x_1, \tilde x_2$ yields a lift of the hyperplane through $x_1, x_2, x_3, x_4$.

Since coplanarity is preserved, we can forget everything except for the plane $x_0=0$ (say) and get a desired lifting of $\PP^2_k$. Equivalently, we could have used a projection from one of the $\FF_p$-rational points.
\end{proof}



To finish, we prove that the matroid $\mathbb{P}^2(\mathbb{F}_p)$ does not admit a projective representation over any ring $A$ with $pA\neq 0$. For $A$ a field, this is well-known (cf. e.g. \cite[\S 2]{gordon}), but we need to make sure that the proof works for arbitrary rings.

\begin{prop} \label{lemma:matroid}
Let $A$ be a ring, $\rho:\PP^2(\FF_p)\to \PP^2(A)$ a map taking triples of collinear points to triples of collinear points. Then $pA=0$.
\end{prop}

\begin{proof}
Changing coordinates in $\PP^2(A)$, we can assume that
\[\rho(1:0:0)=(1:0:0), \quad \rho(0:1:0)=(0:1:0),  \quad \rho(0:0:1)=(0:0:1), \] 
and $\rho(1:1:1) = (1:1:1)$. Thus 
\[ \rho(1:1:0)=\rho\left(\ell((0:0:1),(1:1:1))\cap \ell((1:0:0), (0:1:0))\right)= (1:1:0) \]
as well. For $n\in\ZZ$, let $P_n = (n:0:1)$, $Q_n=(n+1:1:1)\in \PP^2(\FF_p)$, and let $P'_n, Q'_n\in \PP^2(A)$ be the points with the same coordinates as $P_n, Q_n$.   We check by induction on $n\geq 0$ that $\rho(P_n) = P'_n$ and $\rho(Q_n) = Q'_n$: the base case is ok, and for the induction step we note that $P_n = \ell(Q_{n-1}, (0:1:0))\cap \ell(P_0, (1:0:0))$, $Q_n = \ell(P_n, (1:1:0))\cap \ell(Q_0, (1:0:0))$, and that the same statements hold with the primes (see Figure~\ref{figure:matroid}). Thus $(p:0:1) = \rho(p:0:1) = \rho(0:0:1) = (0:0:1)$, and hence $p=0$ in $A$. 
\begin{figure}[ht]
  \label{figure:matroid}
  \centering
\begin{tikzpicture}
 [
    scale=1,
    >=stealth,
    point/.style = {draw, circle,  fill = black, inner sep = 1pt},
    dot/.style   = {draw, circle,  fill = black, inner sep = .2pt},
  ]

\coordinate[point,label=below:$P_{p-1}$] (Pp) at (-1,0);
\coordinate[point,label=below:$P_0$] (P0) at (0,0);
\coordinate[point,label=below:$P_1$] (P1) at (1,0);
\coordinate[point,label=below:$(1:0:0)$] (R) at (7,0);
\coordinate (R1) at (5, 0);
\coordinate (R2) at (5, 1);

\coordinate[point,label=above:$Q_{p-1}$] (Q0) at (0,1);
\coordinate[point,label=above:$Q_0$] (Q1) at (1,1);
\coordinate[point,label=above:$Q_1$] (Q2) at (2,1);


\coordinate (S0) at (0,5);
\coordinate (S1) at (1,5);
\coordinate[point,label=above:$(0:1:0)$] (S) at (0,7);

\coordinate (T0) at (3,3);
\coordinate (T1) at (3.5,2.5);
\coordinate (Tp) at (2.5,3.5);
\coordinate[label={\rm line at infinity}] (I) at (7, 2.5);
\coordinate[point,label=right:$(1:1:0)$] (T) at (5,5);

\draw (P0)--(P1)--(R1) (Q0)--(Q1)--(Q2)--(R2); 
\draw (Pp)--(P0)--(Q1)--(T0) (P1)--(Q2)--(T1) (P0)--(Q0)--(S0) (P1)--(Q1)--(S1) (Pp)--(Q0)--(Tp); 
\draw[dotted] (R1)--(R) (R2)--(R) (T0)--(T) (T1)--(T) (S0)--(S) (S1)--(S) (Tp)--(T);

\draw[dashed] (R) arc (0:90:7);
\end{tikzpicture}
  \caption{\emph{Proof of \cref{lemma:matroid}}}

\end{figure}
\end{proof}

\begin{remark}
Note that the proof exhibits a sub-matroid (denoted $M_p$ in \cite{gordon}) consisting of $2p+3$ points sharing the desired property of $\PP^2(\FF_p)$. This means that in our second non-liftable example we could have blown up a smaller configuration of $2p+4$ points and (strict transforms of) $4p+7$ lines between them. 
\end{remark}

\begin{remark}
With the same proof, one can construct similar examples in higher dimensions: blow up $\PP^n_k$ ($n\geq 3$) in all $\FF_p$-rational points, (strict transforms of) lines, planes, and so on.  Such varieties were studied in \cite[Definition~1.2]{remy_thuillier_werner} in relation to automorphisms of the Drinfeld half-space. 
\end{remark}

\begin{remark}
We also remark that in \cite[Proposition~8.4]{langer} it is proved that a pair $(X,D)$ where $X$ is the blow-up of $\PP^2_k$ in all $\FF_p$-rational points and $D$ is a union of strict transforms of at least $4p-3$ $\FF_p$-rational lines does not lift to $W_2(k)$.  The argument above proves that the matroid $M_p$ leads to a non-liftable example with a fewer number of lines equal to $2p+3$.  We do not know whether $2p+3$ is the minimal number of lines necessary to exhibit $W_2(k)$ non-liftability.
\end{remark}



\section{Technical background}

Here we review the necessary technical results regarding deformation theory of products (\S\ref{sec:deformation}), descending deformations along morphisms (\S\ref{sec:descend}), cohomology of blowing up (\S\ref{sec:cohomology}), and Hodge--de Rham degeneration, ordinarity, and the Hodge--Witt property of blow-ups (\S\ref{sec:ordinarity}).

\subsection{Deformations of products}\label{sec:deformation}

Our goal is to show that given two $k$-schemes $X$ and $Y$ such that $H^1(X, \cO_X) = H^1(Y, \cO_Y) = 0$, then every deformation of $X\times Y$ comes from a pair of deformations of $X$ and $Y$ (\cref{lem:product_deformation}). We begin with a few remarks concerning deformation obstruction classes.  Firstly, observe that by \cite{illusie_cotangent} we know that for any $k$-scheme $Z$ the obstruction class to lifting an element 
\[
\left(\wt{f} : \wt{Z} \to \Spec(A)\right) \in \Def_Z(A)
\]
to a thickening $\eta : 0 \ra I \ra (B,\ideal{m}_B) \ra (A,\ideal{m}_A) \ra 0$, satisfying $\ideal{m}_B I = 0$, is given by a class in $\Ext^2(\Cot_{Z/k},\cO_Z) \otimes_k I$ defined as the Yoneda composition of Kodaira--Spencer class $K_{\wt{Z}/A/\ZZ} \in \Ext^1(\Cot_{\wt{Z}/A},L\wt{f}^*\Cot_{A/\ZZ}[1])$ and the pullback $\wt{f}^*\eta$ of the extension class $\eta \in \Ext^1(\Cot_{A/\ZZ},I)$.  Moreover, by a simple diagram chase based on the properties of the cotangent complex we obtain:

\begin{lemma}[Additivity of Kodaira--Spencer]
Let $f : X \to Z$ and $g : Y \to Z$ be morphisms of $S$-schemes.  Let $p_X : X \times_Z Y \to X$ and $p_Y : X \times_Z Y \to Y$ denote the projections, $h : X \times_Z Y \to Z$ the composition $h = f \comp p_X = g \comp p_Y$. Then, Kodaira--Spencer class: 
\[
K_{{X \times_Z Y}/Z/S} \in \Ext^1(\Cot_{X \times_Z Y/Z},Lh^*\Cot_{Z/S})
\] 
equals the direct sum of pullbacks of Kodaira--Spencer classes: 
\begin{align*}
K_{X/Z/S} \in \Ext^1(\Cot_{X/Z},Lf^*\Cot_{Z/S}) \textrm{ and }
K_{Y/Z/S} \in \Ext^1(\Cot_{Y/Z},Lg^*\Cot_{Z/S}).
\end{align*}
\end{lemma}

Equipped with the above description, we are ready to prove:

\begin{prop}\label{lem:product_deformation}
The morphism of deformation functors 
\[ \proddef_{X,Y} : \Def_X \times \Def_Y \to \Def_{X \times Y},\quad (\wt{X},\wt{Y}) \mapsto \wt{X} \times_{\Spec(A)} \wt{Y} \] 
is smooth (in particular levelwise surjective) if $H^1(X,\cO_X) = H^1(Y,\cO_Y) = 0$.
\end{prop}
\begin{proof}
By the above general considerations and the additivity of Kodaira--Spencer class we see that the morphisms on tangent and obstruction space
\begin{align*}
\Tan_{\proddef_{X,Y}} : \Ext^1(\Cot_{X/k},\cO_X) \oplus \Ext^1(\Cot_{Y/k},\cO_Y) \to \Ext^1(Lp_X^*\Cot_{X/k} \oplus Lp_Y^*\Cot_{Y/k},\cO_{X \times Y} ), \\
\Ob_{\proddef_{X,Y}} : \Ext^2(\Cot_{X/k},\cO_X) \oplus \Ext^2(\Cot_{Y/k},\cO_Y) \to \Ext^2(Lp_X^*\Cot_{X/k} \oplus Lp_Y^*\Cot_{Y/k},\cO_{X \times Y} ), 
\end{align*} 
are given as direct sums of morphisms: 
\begin{align*}
\Ext^\bullet(\Cot_{X/k},\cO_X) \to \Ext^\bullet(\Cot_{X/k},Rp_{X*}\cO_{X \times Y}) \isom \Ext^\bullet(L{p_X}^*\Cot_{X/k},\cO_{X \times Y} ); \\
\Ext^\bullet(\Cot_{Y/k},\cO_Y) \to \Ext^\bullet(\Cot_{Y/k},Rp_{Y*}\cO_{X \times Y}) \isom \Ext^\bullet(Lp_Y^*\Cot_{Y/k},\cO_{X \times Y} ),
\end{align*}
which arise from the natural distinguished triangles
\begin{align*}
\cO_X \ra Rp_{X*}\cO_{X \times Y} \ra C_{p_X} \quad \text{ and } \quad \cO_Y \ra Rp_{Y*}\cO_{X \times Y} \ra C_{p_Y},
\end{align*}
induced by the structure morphisms $p_X^\#$ and $p_Y^\#$ of the projections.

By the assumptions and the spectral sequence:
\[
E^2_{ij} = \Ext^i(\Cot_{X/k},\cH^j(C_{p_X})) \Rightarrow \Ext^{i+j}(\Cot_{X/k},C_{p_X})
\] 
we see that $\Ext^1(\Cot_{X/k},C_{p_X}) = H^1(Y,\cO_Y) \otimes_k \Ext^0(L_{X/k},\cO_X) = 0$.  Analogously we obtain $\Ext^1(\Cot_{Y/k},C_{p_Y}) = 0$.  Therefore $\Tan_{\proddef_{X,Y}}$ is surjective and $\Ob_{\proddef_{X,Y}}$ is injective, which by \cite[Lemma 6.1]{fantechi_manetti} implies that $\proddef_{X,Y}$ is a smooth morphism of deformation functors.
\end{proof}

\subsection{Descending deformations along morphisms} 
\label{sec:descend}

One of our main tools is the following proposition, which one can prove along the same lines as \cite[Proposition 2.2]{liedtke_satriano}. See also \cite{cynk_van_straten} and \cite{wahl}, where this idea appeared previously.

\begin{prop} \label{lem:blow-up}
\begin{enumerate}
\item Let $f:Y \to X$ be a map satisfying $Rf_* \cO_Y = \cO_X$. Then there exists a natural transformation $\Def_Y\to \Def_X$. More generally, if $W=\{W_i\}_{i\in I}$ (resp. $Z=\{Z_i\}_{i\in I}$) is a family of closed subschemes of $Y$ (resp. $X$) parametrized by a preorder $I$ (cf. \S\ref{sec:notation}), and if $Rf_* \cO_{W_i} = \cO_{Z_i}$ (in particular, $Z_i=f(W_i)$), then there exists a natural transformation $\Def_{Y, W} \to \Def_{X, Z}$.  
\item Let $X$ be a smooth scheme, $Z\subseteq X$ a smooth closed subscheme of codimension $\geq 2$, $f:Y=\Bl_Z X\to X$ the blow-up of $X$ along $Z$, $E=\{E_j\}_{j\in J}$ the set of connected components of $f^{-1}(Z)$.  Then the forgetful transformation $\Def_{Y,  E}\to \Def_{Y}$ is an isomorphism (here the index set $J$ is given the trivial order). More generally, if $W=\{W_i\}_{i\in I}$ is a family of closed subschemes of $Y$, then the forgetful transformation $\Def_{Y, W\sqcup E}\to \Def_{Y, W}$ is an isomorphism. Here by $W\sqcup E$ we mean the family $\{W_i\}_{i\in I}\sqcup \{E_j\}_{j\in J}$ parametrized by $I\sqcup J$ with no nontrivial relations between $I$ and $J$.
\end{enumerate}
\end{prop}

As a simple corollary we obtain:

\begin{prop}\label{lem:lifting_graph}
Let $f : X \to Y$ be a morphism of schemes over a field $k$ satisfying $H^1(X,\cO_X) = H^1(Y,\cO_Y) = 0$.  If $\Bl_{\Gamma_f}(X \times Y)$ lifts to $A \in \Art_{W(k)}(k)$, then there exist $A$-liftings of $X$ and $Y$ together with a lifting of $f$.
\end{prop}
\begin{proof}
Assume $\Bl_{\Gamma_f}(X \times Y)$ lifts to $A$.  By \cref{lem:blow-up} there exists a deformation $\wt{X \times Y}$ of the product $X \times Y$ together with an embedded deformation $\wt{\Gamma_f}$ of $\Gamma_f$.  By \cref{lem:product_deformation} the $A$-scheme $\wt{X \times Y}$ is isomorphic to $\wt{X} \times_{\Spec(A)} \wt{Y}$ for some deformations of $X$ and $Y$.  The restriction of the projection $\wt{p}_X : \wt{X} \times_{\Spec(A)} \wt{Y} \to \wt{X}$ to $\wt{\Gamma_f}$ is an isomorphism (as its restricton to $\Spec(k)$ is an isomorphism) and therefore the tuple $(\wt{X},\wt{Y},\wt{p}_Y \comp (\wt{p}_X|_{\wt{\Gamma_f}})^{-1})$ gives the desired pair of liftings of $X$ and $Y$ together with a lifting of $f$. 
\end{proof}

\subsection{Regular sequences and flatness}

In the proof of Theorem~\ref{thm:config}, we need the following simple claim: if $R$ is a $3$-dimensional regular local $k$-algebra with residue field $k$, $L$ and $H$ a smooth curve and a smooth hypersurface in $X=\Spec\, R$ intersecting transversally at the closed point $P$, then any embedded deformation of $(X, L, H)$ induces a deformation of $P$ inside $L$ and $H$. This claim is implied by the following general results regarding deformations and regular sequences.   

\begin{lemma}\label{lem:nonzero}
Suppose $(A,\ideal{m}_A)$ is an element of $\Art_{W(k)}(k)$ and $R$ is a local $k$-algebra.  Moreover let $S$ be an $A$-flat local ring such that $S \otimes_A k = R$.  Then for any element $f \in S$ such that $\ol{f} \in R$ (we denote by $\ol{f}$ the image of $f$ under the natural map $S \to R$) is a non-zero divisor the following assertions hold true:

\begin{enumerate}
\item the element $f$ is a non-zero divisor in $S$,
\item the quotient ring $S/(f)$ is $A$-flat.
\end{enumerate}
\end{lemma}
\begin{proof}
The proof of the first claim follows by induction with respect to the length of $A$.  For the case $\len(A) = 1$ we know that $A = k$ and therefore the claim is clear.  For $\len(A) > 1$, we observe that $(A,\ideal{m}_A)$ is an extension of $(A',\ideal{m}_{A'}) \in \Art_{W(k)}(k)$ of $\len(A') = \len(A) - 1$ by a principal ideal $I = (s)$ satisfying $\ideal{m}_A I = 0$.  Now, take an element $g \in S$ such that $gf = 0$.  By the induction hypothesis applied for $A'$ the element $[f] \in S' \mydef S \otimes_A A'$ (we denote by $[g]$ the image of $g$ under the natural map $S \to S'$) is a non-zero divisor in $S'$ which implies that $[g] = 0$ and therefore there exists an element $g' \in S$ such that $g = sg'$.  Consequently from the relation $sg'f = 0$ we infer that $g'f \in \Ann(s)$ which by the $A$-flatness of $S$ means that $g'f \in \ideal{m}_A S$, that is $\ol{g'} \cdot \ol{f} = 0$.  By the induction hypothesis applied for $A = k$ we see that $g' \in \ideal{m}_A S$ which yields that $g \in \ideal{m}_A I \cdot S = (0)$.  This implies that $f$ is a non-zero divisor and thus proves the first part of the lemma. 

The proof of the second claim is a standard application of local criterion of flatness applied to an $A$-flat resolution:
\[
\xymatrix{
  0 \ar[r] & S \ar[r]^{f \cdot } & S \ar[r] & \ar[r] S/(f) \ar[r] & 0 }
\] implied by the first claim. 
\end{proof}

\begin{cor}\label{cor:reg_seq}
Let the rings $(A,\ideal{m}_A)$, $R$ and $S$ be as above.  Moreover, let $f_1,\ldots,f_k \in S$ for $k \geq 1$ be a sequence of elements such that their reductions $\ol{f_1},\ldots,\ol{f_k} \in R$ form a regular sequence in $R$.  Then $f_1,\ldots,f_k$ is a regular sequence in $S$ and $S/(f_1,\ldots,f_k)$ is an $A$-flat lifting of $R/(\ol{f_1},\ldots,\ol{f_k})$.
\end{cor}
\begin{proof}
The proof follows from \cref{lem:nonzero} by induction with respect to the parameter $k$.
\end{proof}

\subsection{Blow-up formulas}\label{sec:cohomology}

In this section, we review formulas for the cohomology of the blow-up of a smooth proper scheme $X$ along a smooth subscheme $Z$, and deduce statements regarding Hodge--de Rham degeneration, ordinarity, and the Hodge--Witt property. It is best to deduce the blow-up formulas for different cohomology theories from a single motivic statement.

\begin{prop}[{cf. \cite[3.5.3]{voevodsky_motives}}] \label{prop:motivic-blow-up-formula}
Suppose that $X$ is a smooth proper scheme over a field $k$, $Z\subseteq X$ a smooth closed subscheme of codimension $c$. Then there is a decomposition of Chow motives
\[ 
    M(\Bl_Z X) = M(X) \oplus \bigoplus_{i=1}^{c-1} M(Z)(i)[2i]
\]
In particular, $[\Bl_Z X] = [X] + (\LL + \LL^2 + \ldots \LL^{c-1})[Z]$ in the Grothendieck ring of varieties, where $\LL = [\mathbb{A}^1_k]$.
\end{prop}

\begin{cor} \label{prop:blow-up-formula}
Suppose that $X$ is a smooth proper scheme over a field $k$, $Z\subseteq X$ a smooth closed subscheme of codimension $c$. Let $H^n$ denote one of the following families of functors of smooth projective varieties $X$:
\begin{enumerate}[(1)]
  \item $H^n(X\otimes \bar k, \mathbb{Z}_\ell)$ for some $\ell$ invertible in $k$, treated as a ${\rm Gal}(\bar k/k)$-module,
  \item (if $k$ is perfect of characteristic $p>0$) $H^n(X/W(k))$, the integral crystalline cohomology, a $W(k)$-module with a $\sigma$-linear endomorphism induced by the Frobenius,
  \item $H^n_{dR}(X) = H^n(X, \Omega^\bullet_{X/k})$, the algebraic de Rham cohomology, endowed with the Hodge filtration,
  \item $H^n_{Hdg}(X) = \bigoplus_{p+q=n} H^q(X, \Omega^p_{X/k})$, Hodge cohomology, a graded $k$-vector space,
  \item $H^n_{HW}(X) =  \bigoplus_{p+q=n} H^q(X, W\Omega^p_{X})$, Hodge--Witt cohomology, a graded $W(k)$-module endowed with $\sigma^{\pm 1}$-linear endmorphisms $F$ and $V$ satisfying $FV=p=VF$,
  \item (if $n$ is even) $A^{n/2}(X)$, the $\frac{n}{2}$-th Chow group of $X$ (treated as an abelian group).
\end{enumerate}
Moreover, let $-(n)$ denote the Tate twist, i.e., the tensor product with $H^2(\mathbb{P}^1)^{\otimes n}$ in the appropriate tensor category. Then there is a natural isomorphism of objects in the appropriate category as listed above
\[ 
  H^n(\Bl_{Z}(X)) = H^n(X) \oplus \bigoplus_{i=1}^{c-1} H^{n-2i}(Z)(i).
\]
\end{cor}

\begin{proof}
This follows from \cref{prop:motivic-blow-up-formula} and the fact that the cohomology theories $H$ above all admit cycle class maps and actions by correspondences. For $\ell$-adic and crystalline cohomology this is well-known, and for Hodge and Hodge--Witt cohomology it follows from the work of Chatzistamatiou and R\"ulling \cite{rulling_chatzistamatiou}.
\end{proof}

\subsection{Hodge--de Rham degeneration, ordinarity, and the Hodge--Witt property} 
\label{sec:ordinarity}

Let $X$ be a smooth proper scheme over $k$. The first hypercohomology spectral sequence of the de Rham complex $\Omega^\bullet_{X/k}$,
\[ 
  E_1^{ij} = H^j(X, \Omega^i_{X/k}) \quad \Rightarrow \quad H^{i+j}_{dR}(X/k) := H^{i+j}(X, \Omega^\bullet_{X/k}),
\]
is called the \emph{Hodge--de Rham spectral sequence} of $X$. We say that it \emph{degenerates} if it degenerates on the first page, i.e., there are no nonzero differentials. As $X$ is proper, the cohomology groups are finite dimensional, and hence the degeneration is equivalent to the condition that
\begin{equation} \label{eqn:hdrdeg} \sum_n \dim H^n_{dR}(X/k) = \sum_{p,q} \dim H^q(X, \Omega^p_{X/k}).
\end{equation}
The Hodge--de Rham spectral sequence of $X$ degenerates if $k$ is of characteristic zero, or if $\dim X < p=\text{char}\,k$ and $X$ lifts to $W_2(k)$ \cite[Corollaire~2.4]{deligne_illusie}.

The scheme $X$ is called \emph{ordinary} (in the sense of Bloch and Kato) if it the Frobenius $F:H^q(X, W\Omega^p_X)\to H^q(X, W\Omega^p_X)$ on Hodge--Witt cohomology is bijective for all $p$ and $q$ (cf. \cite[Proposition~7.3]{bloch_kato}) for several equivalent criteria). It is called \emph{Hodge--Witt} if the Hodge--Witt groups $H^q(X, W\Omega^p_X)$ are finitely generated $W(k)$-modules. It follows from \cite[IV 4]{illusie_raynaud} that $X$ is Hodge--Witt if it is ordinary, and that $X\times Y$ is ordinary if $X$ and $Y$ are. 

\begin{cor}\label{lem:ordinary_hodge_witt}
Suppose that $X$ is a smooth proper scheme over a field $k$, $Z\subseteq X$ a smooth closed subscheme of codimension $>1$. Then
\begin{enumerate}
  \item The Hodge--de Rham spectral sequences of $Z$ and $X$ degenerate if and only if the Hodge--de Rham sequence of  $\Bl_{Z} X$ degenerates.
  \item The scheme $\Bl_Z(X)$ is ordinary (resp. Hodge--Witt) if and only if both $X$ and $Y$ are ordinary (resp. Hodge--Witt).
\end{enumerate}
\end{cor}

\begin{proof}
The first assertion follows from \cref{prop:blow-up-formula} for $H^n_{dR}$ and $H^n_{Hdg}$ and \eqref{eqn:hdrdeg}. For the latter, use \cref{prop:blow-up-formula}  for $H^n_{HW}$ and the characterizations given above.
\end{proof}

\bibliographystyle{amsalpha} 
\bibliography{graphFrobenius}

%
%


\end{document}